\documentclass[12pt]{amsart}
\usepackage{}
\usepackage{amsmath}
\usepackage{amsfonts}
\usepackage{amssymb}
\usepackage[all,cmtip]{xy}           
\usepackage{shuffle}
\usepackage{bbding}
\usepackage{txfonts}
\usepackage[shortlabels]{enumitem}
\usepackage{ifpdf}
\ifpdf
\usepackage[colorlinks,final,backref=page,hyperindex]{hyperref}
\else
\usepackage[colorlinks,final,backref=page,hyperindex,hypertex]{hyperref}
\fi
\usepackage{tikz}
\usepackage[active]{srcltx}

\usepackage{difftrees}

\makeatletter

\newtheorem{theorem}{Theorem}[section]
\newtheorem{prop}[theorem]{Proposition}
\newtheorem{lemma}[theorem]{Lemma}

\newtheorem{prop-def}{Proposition-Definition}[section]

\theoremstyle{definition}
\newtheorem{defn}[theorem]{Definition}

\newtheorem{remark}[theorem]{Remark}
\newtheorem{exam}[theorem]{Example}

\newcommand{\nc}{\newcommand}

\newcommand {\emptycomment}[1]{}

\nc{\delete}[1]{{}}
\nc{\mmargin}[1]{}

\nc{\mlabel}[1]{\label{#1}}  
\nc{\mcite}[1]{\cite{#1}}  
\nc{\mref}[1]{\ref{#1}}  
\nc{\meqref}[1]{\eqref{#1}}  
\nc{\mbibitem}[1]{\bibitem{#1}} 

\delete{
	\nc{\mlabel}[1]{\label{#1}  
		{\hfill \hspace{1cm}{\bf{{\ }\hfill(#1)}}}}
	\nc{\mcite}[1]{\cite{#1}{{\bf{{\ }(#1)}}}}  
	\nc{\mref}[1]{\ref{#1}{{\bf{{\ }(#1)}}}}  
	\nc{\meqref}[1]{\eqref{#1}{{\bf{{\ }(#1)}}}}  
	\nc{\mbibitem}[1]{\bibitem[\bf #1]{#1}} 
}

\newcommand{\huaH}{\mathcal{H}}

\newcommand{\huaX}{\mathcal{X}}
\newcommand{\huaY}{\mathcal{Y}}
\newcommand{\huaO}{{\mathcal{OT}}}
\newcommand{\huaF}{{\mathcal{OF}}}

\newcommand{\bk}{{\mathbf{k}}}

\nc{\vep}{\varepsilon}
\nc{\bin}[2]{ (_{\stackrel{\scs{#1}}{\scs{#2}}})}  
\nc{\binc}[2]{(\!\! \begin{array}{c} \scs{#1}\\
		\scs{#2} \end{array}\!\!)}  
\nc{\bincc}[2]{  ( {\scs{#1} \atop
		\vspace{-1cm}\scs{#2}} )}  
\nc{\oline}[1]{\overline{#1}}
\nc{\mapm}[1]{\lfloor\!|{#1}|\!\rfloor}
\nc{\bs}{\bar{S}}
\nc{\cast}{{\,\mbox{\raisebox{.8pt}{$\scriptstyle \circledast$}}\,}}
\nc{\la}{\longrightarrow}
\nc{\ot}{\otimes}
\nc{\rar}{\rightarrow}
\nc{\dar}{\downarrow}
\nc{\dap}[1]{\downarrow \rlap{$\scriptstyle{#1}$}}
\nc{\defeq}{\stackrel{\rm def}{=}}
\nc{\dis}[1]{\displaystyle{#1}}
\nc{\dotcup}{\ \displaystyle{\bigcup^\bullet}\ }
\nc{\hcm}{\ \hat{,}\ }
\nc{\hts}{\hat{\otimes}}
\nc{\hcirc}{\hat{\circ}}
\nc{\lleft}{[}
\nc{\lright}{]}
\nc{\curlyl}{\left \{ \begin{array}{c} {} \\ {} \end{array}
	\right .  \!\!\!\!\!\!\!}
\nc{\curlyr}{ \!\!\!\!\!\!\!
	\left . \begin{array}{c} {} \\ {} \end{array}
	\right \} }
\nc{\longmid}{\left | \begin{array}{c} {} \\ {} \end{array}
	\right . \!\!\!\!\!\!\!}
\nc{\ora}[1]{\stackrel{#1}{\rar}}
\nc{\ola}[1]{\stackrel{#1}{\la}}
\nc{\scs}[1]{\scriptstyle{#1}} \nc{\mrm}[1]{{\rm #1}}
\nc{\dirlim}{\displaystyle{\lim_{\longrightarrow}}\,}
\nc{\invlim}{\displaystyle{\lim_{\longleftarrow}}\,}
\nc{\dislim}[1]{\displaystyle{\lim_{#1}}} \nc{\colim}{\mrm{colim}}
\nc{\mvp}{\vspace{0.3cm}} \nc{\tk}{^{(k)}} \nc{\tp}{^\prime}
\nc{\ttp}{^{\prime\prime}} \nc{\svp}{\vspace{2cm}}
\nc{\vp}{\vspace{8cm}}
\nc{\modg}[1]{\!<\!\!{#1}\!\!>}
\nc{\intg}[1]{F_C(#1)}
\nc{\lmodg}{\!<\!\!}
\nc{\rmodg}{\!\!>\!}
\nc{\cpi}{\widehat{\Pi}}
\nc{\labs}{\mid\!}
\nc{\rabs}{\!\mid}
\nc{\btr}{\blacktriangleright}

\nc{\ad}{\mrm{ad}}
\nc{\rRB}{\mathsf{rRB}}
\nc{\cocrRB}{\mathsf{cocrRB}}
\nc{\PH}{\mathsf{PH}}
\nc{\cocPH}{\mathsf{cocPH}}
\nc{\ann}{\mrm{ann}}
\nc{\Aut}{\mrm{Aut}}
\nc{\Der}{\mrm{Der}}
\nc{\Sym}{\mrm{Sym}}
\nc{\br}{\mrm{bre}}
\nc{\can}{\mrm{can}}
\nc{\Cont}{\mrm{Cont}}
\nc{\rchar}{\mrm{char}}
\nc{\cok}{\mrm{coker}}
\nc{\de}{\mrm{dep}}
\nc{\dtf}{{R-{\rm tf}}}
\nc{\dtor}{{R-{\rm tor}}}

\nc{\Dif}{\mrm{Diff}}
\nc{\Div}{\mrm{Div}}
\nc{\End}{\mrm{End}}
\nc{\Ext}{\mrm{Ext}}
\nc{\Fil}{\mrm{Fil}}
\nc{\Fr}{\mrm{Fr}}
\nc{\Frob}{\mrm{Frob}}
\nc{\Gal}{\mrm{Gal}}
\nc{\GL}{\mrm{GL}}
\nc{\Gr}{\mrm{Gr}}
\nc{\Hom}{\mrm{Hom}}
\nc{\Hoch}{\mrm{Hoch}}
\nc{\hsr}{\mrm{H}}
\nc{\hpol}{\mrm{HP}}
\nc{\id}{\mrm{id}}
\nc{\im}{\mrm{im}}
\nc{\inv}{\mrm{inv}}
\nc{\Id}{\mrm{Id}}
\nc{\ID}{\mrm{ID}}
\nc{\Irr}{\mrm{Irr}}
\nc{\incl}{\mrm{incl}}
\nc{\length}{\mrm{length}}
\nc{\NLSW}{\mrm{NLSW}}
\nc{\Lie}{\mrm{Lie}}
\nc{\mchar}{\rm char}
\nc{\mpart}{\mrm{part}}
\nc{\ql}{{\QQ_\ell}}
\nc{\qp}{{\QQ_p}}
\nc{\rank}{\mrm{rank}}
\nc{\rcot}{\mrm{cot}}
\nc{\rdef}{\mrm{def}}
\nc{\rdiv}{{\rm div}}
\nc{\rtf}{{\rm tf}}
\nc{\rtor}{{\rm tor}}
\nc{\res}{\mrm{res}}
\nc{\SL}{\mrm{SL}}
\nc{\Spec}{\mrm{Spec}}
\nc{\tor}{\mrm{tor}}
\nc{\Tr}{\mrm{Tr}}
\nc{\tr}{\mrm{tr}}
\nc{\wt}{\mrm{wt}}

\nc{\bfk}{{\bf k}}
\nc{\bfone}{{\bf 1}}
\nc{\bfzero}{{\bf 0}}
\nc{\detail}{\marginpar{\bf More detail}
	\noindent{\bf Need more detail!}
	\svp}
\nc{\gap}{\marginpar{\bf Incomplete}\noindent{\bf Incomplete!!}
	\svp}
\nc{\FMod}{\mathbf{FMod}}
\nc{\Int}{\mathbf{Int}}
\nc{\Mon}{\mathbf{Mon}}
\nc{\remarks}{\noindent{\bf Remarks: }}
\nc{\Rep}{\mathbf{Rep}}
\nc{\Rings}{\mathbf{Rings}}
\nc{\Sets}{\mathbf{Sets}}
\nc{\Diff}{\mathbf{Diff}}
\nc{\Inte}{\mathbf{Inte}}
\nc{\U}{\mathrm{U}}
\newcommand{\Ten}{\mathsf{T}}

\nc{\BA}{{\mathbb A}}   \nc{\CC}{{\mathbb C}}
\nc{\DD}{{\mathbb D}}   \nc{\EE}{{\mathbb E}}
\nc{\FF}{{\mathbb F}}   \nc{\GG}{{\mathbb G}}
\nc{\HH}{{\mathbb H}}   \nc{\LL}{{\mathbb L}}
\nc{\NN}{{\mathbb N}}   \nc{\PP}{{\mathbb P}}
\nc{\QQ}{{\mathbb Q}}   \nc{\RR}{{\mathbb R}}
\nc{\TT}{{\mathbb T}}   \nc{\VV}{{\mathbb V}}
\nc{\ZZ}{{\mathbb Z}}   \nc{\TP}{\widetilde{P}}

\nc{\cala}{{\mathcal A}}    \nc{\calc}{{\mathcal C}}
\nc{\cald}{\mathcal{D}}     \nc{\cale}{{\mathcal E}}
\nc{\calf}{{\mathcal F}}    \nc{\calg}{{\mathcal G}}
\nc{\calh}{{\mathcal H}}    \nc{\cali}{{\mathcal I}}
\nc{\call}{{\mathcal L}}    \nc{\calm}{{\mathcal M}}
\nc{\caln}{{\mathcal N}}    \nc{\calo}{{\mathcal O}}
\nc{\calp}{{\mathcal P}}    \nc{\calr}{{\mathcal R}}

\nc{\cals}{{\mathcal S}}    \nc{\calt}{{\Omega}}
\nc{\calv}{{\mathcal V}}    \nc{\calw}{{\mathcal W}}
\nc{\calx}{{\mathcal X}}

\nc{\fraka}{{\mathfrak a}}
\nc{\frakb}{\mathfrak{b}}
\nc{\frakg}{{\frak g}}
\nc{\frakl}{{\frak l}}
\nc{\fraks}{{\frak s}}
\nc{\frakB}{{\frak B}}
\nc{\frakm}{{\frak m}}
\nc{\frakM}{{\frak M}}
\nc{\frakp}{{\frak p}}
\nc{\frakW}{{\frak W}}
\nc{\frakX}{{\frak X}}
\nc{\frakS}{{\frak S}}
\nc{\frakA}{{\frak A}}
\nc{\frakx}{{\frakx}}

\nc{\ynr}[1]{\textcolor{orange}{\underline{Yunnan:}#1 }}

\nc{\lir}[1]{\textcolor{red}{\underline{Li:}#1 }}

\delete{
	\input cyracc.def

	\newtheorem{theorem}{Theorem}[section]
	\newtheorem{lemma}[theorem]{Lemma}

	\theoremstyle{definition}

	\theoremstyle{remark}
	\newtheorem{remark}[theorem]{Remark}

	\allowdisplaybreaks
	
	\numberwithin{equation}{section}
	
}

\begin{document}

\title[The sub-adjacent structure on the universal envelop of a post-Lie algebra]{On the sub-adjacent Hopf algebra of the universal enveloping algebra of a post-Lie algebra}

\author{Yunnan Li}
\address{School of Mathematics and Information Science, Guangzhou University,
Guangzhou 510006, China}
\email{ynli@gzhu.edu.cn}

\begin{abstract}
Recently the notion of post-Hopf algebra was introduced, with the universal enveloping algebra of a post-Lie algebra as the fundamental example. A novel property is that any cocommutative post-Hopf algebra gives rise to a sub-adjacent Hopf algebra with a generalized Grossman-Larson product.
By twisting the post-Hopf product, we provide a combinatorial antipode formula for the sub-adjacent Hopf algebra of the universal enveloping algebra of a post-Lie algebra.
Relating to such a sub-adjacent Hopf algebra, we also obtain a closed inverse formula for the Guin-Oudom isomorphism in the context of post-Lie algebras.
Especially as a byproduct, we derive a cancellation-free antipode formula for the Grossman-Larson Hopf algebra of ordered trees through a concrete tree-grafting expression.
\end{abstract}

\keywords{post-Hopf algebra, post-Lie algebra, the Grossman-Larson Hopf algebra, ordered tree\\
\qquad 2020 Mathematics Subject Classification. 16T30, 17B35, 17D25, 05C05}

\maketitle

\tableofcontents

\allowdisplaybreaks

\section{Introduction}

The notion of post-Lie algebra was introduced by Vallette in his study of operad theory \cite{Val}. Such an algebraic structure has important applications in the study of classical Yang-Baxter equation \cite{BGN}, geometric numerical integration \cite{Fard1,Fard2,ML},
and plays an important role recently in the theory of regularity structures and planarly branched rough paths for solving SPDEs \cite{BK,CEMM}. A pre-Lie algebra, seen as a post-Lie algebra in which the Lie algebra is abelian \cite{CL} has its historical origin in algebraic deformation theory, geometry and physics; see~\cite{Bu,Ma} for detailed reviews.

Since then the studies of the universal enveloping algebras of pre-Lie algebras as well as post-Lie algebras have attracted considerable attention. The remarkable work \cite{OG} of Oudom and Guin defined a functor from the category of pre-Lie algebras to that of Hopf algebras by extending the pre-Lie product of a pre-Lie algebra to the symmetric algebra over it.
In \cite{ELM} Ebrahimi-Fard, Lundervold and Munthe-Kaas further extended the Guin-Oudom construction to post-Lie algebras; see also an operadic approach in~\cite{Ra}. Notably, they found a new Hopf algebra structure assembled from the universal enveloping algebra $\U(\frakg)$ of a post-Lie algebra  $(\frakg,\rhd)$. We call it the sub-adjacent Hopf algebra of $\U(\frakg)$ and denote it by $\U(\frakg)_\rhd$.
Recently, the Guin-Oudom construction has become a crucial tool in constructing combinatorial Hopf algebras used to describe rough paths and regularity structures \cite{BK,ER,Ra0}. For a proper definition of combinatorial Hopf algebras, we refer to~\cite[Definition~3.3]{CEMM}.

Also, a post-Lie algebra $\frakg$ has a sub-adjacent Lie algebra structure $\frakg_\rhd$.
A Hopf algebra isomorphism between the universal enveloping algebra $\U(\frakg_\rhd)$ and the sub-adjacent Hopf algebra  $\U(\frakg)_\rhd$ was given in \cite{ELM}, and later shown in \cite{EMM} to be related to the factorization of classical $r$-matrices on Lie algebras. We call such an isomorphism the Guin-Oudom isomorphism in the context of post-Lie algebras.
Particularly, the inverse of the Guin-Oudom isomorphism for the free post-Lie algebra over a magma algebra is identified with Gavrilov's $K$-map introduced in his study \cite{Gav1} on covariant derivatives and the double exponential; see \cite{AEMM,Foissy} for details.

For broader applications of the sub-adjacent Hopf algebra $\U(\frakg)_\rhd$ from a post-Lie algebra $(\frakg,\rhd)$,
we wonder an explicit description for its two main ingredients, namely
\begin{enumerate}[(1)]
\item the  product $\circ $ of $\U(\frakg)_\rhd$,

\item the antipode map $S_\rhd$ of $\U(\frakg)_\rhd$.

\end{enumerate}
The first one generalizes the classical  Grossman-Larson product defined in the context of trees and has important applications in the studies of Magnus expansions~\cite{EMQ} and Lie-Butcher series~\cite{ML,MW}. It was well-studied in~\cite{ELM}.
In this paper we consider the second one and, in particular, the inverse of the Guin-Oudom isomorphism $\phi:\U(\frakg_\rhd)\to\U(\frakg)_\rhd$ from a different perspective, namely the post-Hopf algebra structure.

The notion of post-Hopf algebra was recently introduced in~\cite{LST} by the author, Sheng and Tang, with the universal enveloping algebra of a post-Lie algebra as the prototype.
The study of post-Hopf algebras was initiated from the innovative concept of Rota-Baxter operators on groups introduced in~\cite{GLS}.
Goncharov succeeded in defining Rota-Baxter operators on cocommutative Hopf algebras in~\cite{Go}. One motivation to propose post-Hopf algebra in~\cite{LST} is to generalize Goncharov's Rota-Baxter Hopf algebras. It also generalizes the notion of $D$-bialgebra in~\cite{MQS} for the study of integration of post-Lie algebras, and is further applied to obtain the formal integration
of connected complete post-Lie algebras, namely an intrinsic post-group structure~\cite{BGST}.
Significantly, any cocommutative  post-Hopf algebra gives rise to a sub-adjacent Hopf algebra with a generalized Grossman-Larson product.

In \cite{GL} Grossman and Larson introduced a family of combinatorial Hopf algebras of trees associated with data structures and used it to give a new proof of Cayley's enumeration of finite rooted trees. The construction of the Grossman-Larson Hopf algebra was also motivated by some concern about differential operators and differential equations, in particular the Runge-Kutta method in numerical analysis~\cite{GL1}.

The Grossman-Larson Hopf algebra of trees has close relationship with many other combinatorial Hopf algebras from different application areas. For example, authors in \cite{Ho}, \cite{Pan} observed that the Grossman-Larson Hopf algebra of (non-planar) rooted trees is graded dual to the Connes-Kreimer Hopf algebra arising from the study on renormalization in perturbative quantum field theory~\cite{CK}. On the other hand, the Grossman-Larson Hopf algebra of ordered trees is graded dual to the Munthe-Kaas-Wright Hopf algebra of ordered forests defined in the framework of Lie groups acting on manifolds and Lie group integrators~\cite{MW,ER}.

In ~\cite[\S 3.4]{BO} Berland and Owren reviewed the antipode of the Grossman-Larson Hopf algebra of ordered trees, and said that an explicit formula for it seems hard to derive in general.
When basic operations of combinatorial objects only define bialgebra structures, to obtain Hopf algebra structures, one may appeal to Takeuchi's celebrated
result~\cite{Ta}, which states that every graded, connected bialgebra is a Hopf algebra
with a universal antipode formula. However, this formula is given by an alternating sum usually involving a large
number of cancellations, so is not optimal for producing combinatorial identities
among the elements of the Hopf algebra. For a combinatorial Hopf algebra, it is more meaningful to write down a cancellation-free antipode formula with a reduced expansion in terms of a combinatorial basis. Here we intend to
provide a cancellation-free antipode formula for the Grossman-Larson Hopf algebra of ordered trees, by serving it as the sub-adjacent Hopf algebra of the universal enveloping algebra of the free post-Lie algebra on one generator. Concrete examples verify its computational efficiency.

So far a substantial body of work has focused on developing cancellation-free antipode formulas for various combinatorial Hopf algebras, including the Hopf algebra of graphs~\cite{HM,CV}, permutations~\cite{XY}, set operads~\cite{MeL},
 and also several variations of quasisymmetric functions~\cite{BS,Li,Pa,YGT}, etc.
For combinatorial Hopf algebras in mathematical physics, we mention the following works.
In \cite{FG,FG1} Figueroa and Gracia-Bondia calculated the antipode of Kreimer's Hopf algebra of Feynman diagrams by the Zimmermann forest formula. In \cite{MP} Menous and Patras generalized such a forest formula to compute the cancellation-free antipode formula for arbitrary right-handed polynomial Hopf algebras, which are graded dual to the universal enveloping algebras of pre-Lie algebras.
Recently, Foissy obtained an explicit antipode formula for the Connes-Moscovici Hopf algebra of trees from the point of view of Com-PreLie Hopf algebras~\cite{Fo}.

The paper is organized as follows. In Section~\ref{sec:pH}, we first recall the notion of post-Hopf algebra and basic properties of a cocommutative post-Hopf algebra, then turn to post-Lie algebras. We focus on the tensor Hopf algebra over a magma algebra, as the universal enveloping algebra of a free post-Lie algebra.
By twisting the post-Hopf product,
we provide a combinatorial antipode formula for its sub-adjacent Hopf algebra (Propositions~\ref{prop:antipode} and~\ref{prop:btr}).
Next we prove that such a formula indeed works for the universal enveloping algebra of an arbitrary post-Lie algebra.
(Theorem~\ref{thm:antipode_uea}).
In Section~\ref{sec:OG}, we first use the same twisted product to describe Gavrilov's $K$-map combinatorially (Proposition~\ref{prop:K-map'}), then deduce a closed inverse formula for the Guin-Oudom isomorphism in the context of post-Lie algebras (Theorem~\ref{thm:hopf-iso-inv}).
In Section~\ref{sec:GL}, we specifically figure out a cancellation-free antipode formula for the Grossman-Larson Hopf algebra of ordered trees (Theorem~\ref{thm:GL_antipode}) together with a concrete tree-grafting interpretation of the corresponding twisted post-Hopf product (Proposition~\ref{prop:GL_graft}).

\vspace{2mm}

\noindent
{\bf Convention.}
In this paper, we fix an algebraically closed ground field $\bk$ of characteristic 0.
Denote $[n]=\{1,\dots, n\}$ for any positive integer $n$.
All the objects under discussion, including vector spaces, algebras and tensor products, are taken over $\bk$ by default.

For any unital algebra $(A,m,u)$ with the multiplication $m$ and the unit map $u:\bk\to A$, let $m^{(0)}=\id$ and for $n\geq1$ we write
$$m^{(n)}=(m\otimes \id^{\otimes(n-1)})\cdots (m\otimes\id)m.$$

For any coalgebra $(C,\Delta,\vep)$, we compress the Sweedler notation of the comultiplication $\Delta$ as $$\Delta(x)=x_1\otimes x_2$$ for simplicity.
Furthermore, let $\Delta^{(0)}=\id$ and for $n\geq1$ we write
$$\Delta^{(n)}(x)=(\Delta\otimes\id^{\otimes (n-1)})\cdots(\Delta\otimes\id)\Delta(x)=x_1\otimes\cdots\otimes x_{n+1}.$$

By convention, a Hopf algebra is denoted by $H=(H,\cdot\,,1,\Delta,\vep,S)$.
For other basic notions of Hopf algebras, we follow the textbooks~\mcite{Mon}.

\section{The universal enveloping algebra of a post-Lie algebra}\label{sec:pH}

First we recall the notion of post-Hopf algebra introduced in~\cite{LST}.
The universal enveloping algebra of a post-Lie algebra is naturally a cocommutative post-Hopf algebra, which gives rise to a sub-adjacent Hopf algebra.
We will deal with the free post-Lie algebra over a magma algebra to derive a combinatorial formula for the antipode of the corresponding sub-adjacent Hopf algebra, and then induce it to the general situation for an arbitrary post-Lie algebra.

\subsection{Post-Hopf algebras and their basic properties}

\begin{defn}\label{defi:pH}
A {\bf post-Hopf algebra} is a pair $(H,\rhd)$, where $H$ is a Hopf algebra and $\rhd:H\otimes H\to H$ is a coalgebra homomorphism satisfying the following equalities:
\begin{eqnarray}
\label{Post-2}x\rhd (y\cdot z)&=&(x_1\rhd y)\cdot(x_2\rhd z),\\
\label{Post-4}x\rhd (y\rhd z)&=&\big(x_1\cdot(x_2\rhd y)\big)\rhd z,
\end{eqnarray}
for any $x,y,z\in H$, and the left multiplication $\alpha_\rhd:H\to \End(H)$ defined by
$$\alpha_{\rhd, x} y= x\rhd y,\quad\forall x,y\in H,$$
is convolution invertible in $\Hom(H,\End(H))$. Namely, there exists a unique linear map $\beta_\rhd:H\to\End(H)$ such that
\begin{equation}\label{Post-con}
\alpha_{\rhd,x_1}\beta_{\rhd,x_2}=\beta_{\rhd,x_1}\alpha_{\rhd,x_2}=\vep(x)\id_H,\quad\forall x\in H.
\end{equation}
\end{defn}

Moreover, we have the following properties.

\begin{lemma}[{\cite[Lemma~2.4]{LST}}]
Let $(H,\rhd)$ be a post-Hopf algebra. Then for all $x,y\in H$, we have
\begin{eqnarray}
\label{Post-1}x\rhd 1&=&\vep(x)1,\\
\label{Post-3}1\rhd x&=&x,\\
\label{Post-5}S(x\rhd y)&=&x\rhd S(y).
\end{eqnarray}
\end{lemma}

\begin{theorem}[{\cite[Theorem~2.5]{LST}}]\label{thm:subHopf}
Let $(H,\rhd)$ be a cocommutative post-Hopf algebra. Then
$$H_\rhd\coloneqq(H,\circ ,1,\Delta,\vep,S_\rhd)$$
is a Hopf algebra, which is called the {\bf sub-adjacent Hopf algebra}, where for all $x,y\in H$,
\begin{eqnarray}
\label{post-rbb-1}x \circ  y&\coloneqq&x_1\cdot (x_2\rhd y),\\
\label{post-rbb-2}S_\rhd(x)&\coloneqq&\beta_{\rhd,x_1}(S(x_2)).
\end{eqnarray}

Furthermore, $H$ is a left $H_\rhd$-module bialgebra via the action $\rhd$.

\end{theorem}

\begin{remark}
(i) The product \eqref{post-rbb-1} generalizes the Grossman-Larson product \cite{GL,MW,OG} defined in the context of (noncommutative) polynomials of ordered trees and widely applied to the study of Magnus expansions \cite{EMQ} and Lie-Butcher series \cite{ML,MW}. A proper generalization of the Grossman-Larson product to the universal enveloping algebra of a post-Lie algebra was already presented in \cite{ELM}.

(ii) Since $\beta_\rhd$ is the convolution inverse of $\alpha_\rhd$ and
$$\alpha_{\rhd,x_1}\alpha_{\rhd,S_\rhd(x_2)}\stackrel{\eqref{Post-4},\,\eqref{post-rbb-1}}{=}
\alpha_{\rhd,x_1\circ  S_\rhd(x_2)}=\alpha_{\rhd,\vep(x)1}\stackrel{\eqref{Post-3}}{=}\vep(x)\id,$$
we have $\beta_{\rhd,x}=\alpha_{\rhd,S_\rhd(x)}$. Namely, Eq.~\eqref{post-rbb-2} is rewritten as
\begin{equation}\label{post-rbb-2'}
S_\rhd(x)=S_\rhd(x_1)\rhd S(x_2),\quad\forall x\in H.
\end{equation}
\end{remark}

\subsection{The universal enveloping algebra of a post-Lie algebra as a post-Hopf algebra}
Recall from \cite{Val} that a {\bf post-Lie algebra} $(\frakg,[\cdot,\cdot]_\frakg,\rhd)$ consists of a Lie algebra $(\frakg,[\cdot,\cdot]_\frakg)$ and a binary product $\rhd:\frakg\otimes\frakg\to\frakg$ such that
\begin{eqnarray}
\label{Post-L-1}x\rhd[y,z]_\frakg&=&[x\rhd y,z]_\frakg+[y,x\rhd z]_\frakg,\\
\label{Post-L-2}([x,y]_\frakg+x\rhd y-y\rhd x)\rhd z&=&x\rhd(y\rhd z)-y\rhd(x\rhd z).
\end{eqnarray}
When $(\frakg,[\cdot,\cdot]_\frakg)$ is abelian, $(\frakg,[\cdot,\cdot]_\frakg,\rhd)$ is a {\bf pre-Lie algebra}.
Any post-Lie algebra $(\frakg,[\cdot,\cdot]_\frakg,\rhd)$ has a {\bf sub-adjacent Lie algebra}
$\frakg_\rhd:=(\frakg,[\cdot,\cdot]_{\frakg_\rhd})$
defined by
\begin{equation}\label{eq:post-L-sub}
[x,y]_{\frakg_\rhd} \coloneqq x\rhd y-y\rhd x+[x,y]_\frakg,\quad\forall x,y\in\frakg.
\end{equation}

The universal enveloping algebra of a pre-Lie algebra and also that of a post-Lie algebra were successively studied in \cite{OG,ELM}. By \cite[Proposition~3.1, Theorem~ 3.4]{ELM}, the binary product $\rhd$ in a post-Lie algebra $(\frakg,[\cdot,\cdot]_\frakg,\rhd)$ can be extended to the universal enveloping algebra $\U(\frakg)$ and induces the sub-adjacent Hopf algebra $\U(\frakg)_\rhd$, which is naturally isomorphic to the universal enveloping algebra $\U(\frakg_\rhd)$ of the sub-adjacent Lie algebra $\frakg_\rhd$ by the Milnor-Moore structure theorem for connected cocommutative Hopf algebras.

We summarize the aforementioned result in the setting of post-Hopf algebras as follows.
\begin{theorem}[\cite{ELM,OG}]\label{th:post-uea}
Let $(\frakg,[\cdot,\cdot]_\frakg,\rhd)$ be a post-Lie algebra associated with the sub-adjacent Lie algebra $\frakg_\rhd$. Then  $(\U(\frakg), \rhd)$ is a post-Hopf algebra with the product $\rhd$ recursively defined by
\begin{align*}
&x\rhd 1=0,\quad x\rhd x_1\cdots x_r=\sum_{i=1}^r x_1\cdots (x\rhd x_i)\cdots x_r,\\
&1\rhd u = u,\quad x_1\cdots x_r\rhd u = x_1\rhd(x_2\cdots x_r\rhd u)-(x_1\rhd x_2\cdots x_r)\rhd u
\end{align*}
for all $x,x_1,\dots,x_r\in \frakg$ and $u\in \U(\frakg)$ with $r\geq1$.

Moreover, there is a Hopf algebra isomorphism $\phi:\U(\frakg_\rhd)\to \U(\frakg)_\rhd$ defined by $$\phi(x_1\cdots x_r)=x_1\circ \cdots\circ  x_r,\quad \forall x_1,\dots,x_r\in \frakg.$$
We call $\phi$ {\bf the Guin-Oudom isomorphism}.
\end{theorem}

\begin{remark}The theory of classical $r$-matrices provides a rich source of examples of post-Lie algebras.
A linear operator $R$ on a Lie algebra $(\frakg,[\cdot,\cdot])$ is called a {\bf classical $r$-matrix} on $\frakg$, if
$$[x,y]_R\coloneqq \dfrac{1}{2}([Rx,y]+[x,Ry]),\quad\forall x,y\in \frakg$$
defines another Lie bracket on $\frakg$. Then such a pair $(\frakg,R)$ is called a {\bf double Lie algebra}~\cite{STS}.
If a classical $r$-matrix $R$ on $\frakg$ satisfies the {\bf modified classical Yang-Baxter equation}
$$[Rx,Ry]=R([Rx,y]+[x,Ry])-[x,y],\quad\forall x,y\in \frakg,$$
then $R_-\coloneqq \frac{1}{2}(R-\id)$ is a Rota-Baxter operator of weight 1, and naturally induces a post-Lie algebra $(\frakg,[\cdot,\cdot],\rhd)$ by
letting $x\rhd y\coloneqq [R_-(x),y]$ for any $x,y\in\frakg$,
such that the sub-adjacent Lie algebra $\frakg_\rhd$ is equal to $\frakg_R=(\frakg,[\cdot,\cdot]_R)$; see~\cite[Corollary~5.6]{BGN}.

Correspondingly, the authors in \cite{EMM} showed that the Guin-Oudom isomorphism $\phi:\U(\frakg_\rhd)\to\U(\frakg)_\rhd$ coincides with the linear isomorphism $F:\U(\frakg_R)\to\U(\frakg)$ in \cite{RS} which arose from the factorization of classical $r$-matrix $R$.
\end{remark}

The above extension from $(\frakg,\rhd)$ to $(\U(\frakg),\rhd)$ can be realized first on the tensor algebra $\Ten(\frakg)$, and then induced to its quotient $\U(\frakg)$.
Especially the universal enveloping algebra of the free post-Lie algebra over a magma algebra is  explicitly constructed by Foissy in \cite{Foissy} as follows.
\begin{prop}
\label{magma-to-post}
Let $(V,\rhd)$ be a magma algebra. Extend the magma operation $\rhd:V \otimes V\to V$
on the vector space $V$
to the coshuffle Hopf algebra $(\Ten(V),\cdot,\Delta_\shuffle,S)$ (using the same notation $\rhd$) as follows:
\begin{eqnarray*}\label{post-brace-algebra-1}
1\rhd a&=&a,\\
x\rhd a&=&x\rhd a,\\
(x\otimes x_1) \rhd a&=&x\rhd (x_1\rhd a)-(x\rhd x_1)\rhd a,\\
&\vdots&\\
(x\otimes x_1\otimes\cdots \otimes x_n) \rhd a&=&x\rhd \big((x_1\otimes\cdots \otimes x_n)\rhd a\big)\\
&&-\sum_{i=1}^{n}\big(x_1\otimes\cdots \otimes x_{i-1}\otimes(x\rhd x_i) \otimes x_{i+1}\otimes\cdots\otimes x_n\big)\rhd a,
\end{eqnarray*}
and
\begin{eqnarray*}\label{post-brace-algebra-2}
1\rhd 1&=&1,\\
x\rhd 1&=&0,\\
X \rhd (a_1\otimes\cdots\otimes a_m)&=&(X_1\rhd a_1)\otimes\cdots\otimes(X_m\rhd a_m),
\end{eqnarray*}
where $x,x_1,\dots,x_n,a,a_1,\dots,a_m\in V$ and $X\in \Ten(V),~\Delta_\shuffle^{(m-1)}(X)=X_1\otimes\cdots\otimes X_m$.
Then the quintuple
$(\Ten(V),\cdot\,,\Delta_\shuffle,S,\rhd)$
is the universal enveloping algebra of the free post-Lie algebra over the magma algebra $(V,\rhd)$, thus a post-Hopf algebra.
\end{prop}

\noindent
{\bf Convention.} From now on, we call an element in $V$ a {\bf letter}, and write any pure tensor in $\Ten(V)$
as a monomial with the tensor product $\otimes$ omitted for simplicity, so call it a {\bf word}.

\subsection{A combinatorial antipode formula for the sub-adjacent structure}
Consider the sub-adjacent Hopf algebra
$$\Ten(V)_\rhd=(\Ten(V),\circ ,\Delta_\shuffle,S_\rhd)$$
of the post-Hopf algebra $(\Ten(V),\cdot,\Delta_\shuffle,S,\rhd)$ mentioned in Proposition~\ref{magma-to-post}.
Recall that the coshuffle coproduct $\Delta_\shuffle$ is defined by $\Delta_\shuffle(x)=x\otimes 1+1\otimes x$ for any $x\in V$. Correspondingly, the iterated coshuffle coproduct on $\Ten(V)$ is given by
\begin{equation}\label{eq:cosh_coproduct}
\Delta_\shuffle^{(r-1)}(x_1\cdots x_m) =
\sum_{\pi:|\pi|=r} x_{B_1}\otimes\cdots\otimes x_{B_r},\quad \forall x_1,\dots, x_m\in V,\ r\geq1,
\end{equation}
where the sum ranges over all partitions $\pi$ of $[m]$ into a tuple $(B_1,\dots,B_{r})$ of $r$ possibly empty subsets, and we use the notation $x_I=x_{i_1}\cdots x_{i_r}\in \Ten(V)$ for any $I=\{i_1<\cdots< i_r\}\subseteq[m]$. In particular, $x_{\emptyset}=1$.
On the other hand, the antipode $S$ is given by
\begin{equation}\label{eq:antipo}
S(x_1 x_2 \cdots x_m)=(-1)^m x_m x_{m-1} \cdots x_1,\quad\forall x_1,x_2,\dots,x_m\in V.
\end{equation}

For the antipode $S_\rhd$ of the sub-adjacent Hopf algebra $\Ten(V)_\rhd$, we give our first result.
\begin{theorem}
For any word $X\in \Ten(V)$ with at most $m$ letters in $V$ $(m\geq1)$,
\begin{equation}\label{post-rbb-2''}
S_\rhd(X)=(\cdots(S(X_1)\rhd S(X_2))\rhd \cdots )\rhd S(X_m).
\end{equation}
\end{theorem}
\begin{proof}
First note that $S_\rhd(x)=S(x)=-x$ for any $x\in V$. So it is done for $m=1$. When $m>1$, we can repeatedly apply Eq.~\eqref{post-rbb-2'} $m-1$ times
to obtain that
$$S_\rhd(X)=(\cdots(S_\rhd(X_1)\rhd S(X_2))\rhd \cdots )\rhd S(X_m).$$
Since $X$ has at most $m$ letters, any term in the iterated coproduct $\Delta_\shuffle^{(m-1)}(X)=X_1\otimes\cdots\otimes X_{m}$ has
all tensor factors from $V$ or at least one tensor factor from $\bk$.
Then by Eq.~\eqref{Post-1}, the leftmost tensor factor $X_1$ should be from $\bk$ or $V$ to make the term survive in the right-hand side of the above equality. Hence, we get the desired formula for the antipode $S_\rhd$ of $\Ten(V)_\rhd$.
\end{proof}

\begin{exam}\label{ex:anti}
Using Eq.~\eqref{post-rbb-2''}, one can easily calculate $S_\rhd(x_1 \cdots x_m)$ with $x_1,\dots,x_m\in V$ for small $m$. We compute the results for $m=2,3$ as a warm-up.
\begin{eqnarray*}
S_\rhd(x_1x_2) &=& S(x_1x_2)+S(x_1)\rhd S(x_2)+S(x_2)\rhd S(x_1)\\
&=& x_2x_1+x_1\rhd  x_2+ x_2\rhd x_1,\\[.5em]
S_\rhd(x_1x_2x_3) &=& S(x_1x_2x_3)+S(x_1)\rhd S(x_2x_3)
+S(x_2)\rhd S(x_1x_3)+S(x_3)\rhd S(x_1x_2)\\
&& + S_\rhd(x_1x_2)\rhd S(x_3) +S_\rhd(x_1x_3)\rhd S(x_2) + S_\rhd (x_2x_3)\rhd S(x_1)\\
&\stackrel{\eqref{Post-4}}{=}& -x_3x_2x_1
- (x_1 \rhd x_3)x_2 - x_3(x_1 \rhd x_2)\\
&& - (x_2 \rhd x_3)x_1 - x_3(x_2 \rhd x_1)
- (x_3 \rhd x_2)x_1 - x_2(x_3 \rhd x_1)\\
&& -(x_1\rhd x_2)\rhd  x_3 -x_2\rhd (x_1\rhd  x_3)
-(x_1\rhd x_3)\rhd  x_2 \\
&& -x_3\rhd (x_1\rhd  x_2)
-(x_2\rhd x_3)\rhd  x_1
-x_3\rhd (x_2\rhd  x_1).
\end{eqnarray*}
\end{exam}

Inspired by Eq.~\eqref{post-rbb-2'}, we introduce the following product twisting the post-Hopf product $\rhd$ by the antipode $S_\rhd$, which is a crucial tool for our study below.
\begin{defn}
Define a binary product $\btr$ on $\Ten(V)$ as follows,
\begin{equation}\label{eq:btr}
X \btr Y\coloneqq (-1)^{|X|}\beta_{\rhd,X}Y=(-1)^{|X|}S_\rhd(X)\rhd Y
\end{equation}
for any words $X,\,Y\in \Ten(V)$, with $|X|$ as the number of letters in $X$, i.e. its natural degree in $\Ten(V)$.
\end{defn}
As we will see in Lemma~\ref{lem:btr_recursive}, the sign $(-1)^{|X|}$ in Eq.~\eqref{eq:btr} is indispensable for the computation to avoid cancellations. Since the antipode $S_\rhd$ is a coalgebra endomorphism of the graded coalgebra $\Ten(V)$, Eq.\eqref{Post-2} implies that
$$X \btr YZ = (X_1 \btr Y)(X_2 \btr Z)$$
for any words $X,Y,Z\in\Ten(V)$. Correspondingly,
\begin{equation}\label{eq:btr_coalg}
X \btr Y = (X_1 \btr y_1)\cdots(X_n \btr y_n),
\end{equation}
when $Y=y_1\cdots y_n$ with $y_1,\dots,y_n\in V$.

\begin{prop}\label{prop:antipode}
The following formula holds for the antipode $S_\rhd$ of $\Ten(V)_\rhd$,
\begin{equation}\label{eq:antipode}
S_\rhd(x_1\cdots x_m)=(-1)^m\sum_\pi (x_{B_1}\btr x_{b_1})\cdots(x_{B_{|\pi|-1}}\btr x_{b_{|\pi|-1}}),\quad \forall x_1,\dots, x_m\in V,
\end{equation}
where the sum ranges over all partitions $\pi$ of $[m]$ into a tuple $(B_1,\dots,B_{|\pi|})$ of possibly empty subsets such that
$B_{|\pi|}=\{b_1>\cdots>b_{|\pi|-1}\}$.
\end{prop}

\begin{proof}
For any word $Y$ in $\Ten(V)$, let $Y^{\rm op}$ be its reverse. By formula~\eqref{eq:cosh_coproduct} of the iterated coshuffle coproduct, we have
\begin{eqnarray*}
S_\rhd(x_1\cdots x_m)&\stackrel{\eqref{post-rbb-2'}}{=}&
\sum_{\pi:|\pi|=2} S_\rhd(x_{B_1})\rhd S(x_{B_2})\\
&\stackrel{\eqref{Post-1},\,\eqref{eq:antipo}}{=}&
\sum_{\pi:|\pi|=2,\,B_2\neq\emptyset}(-1)^{|B_2|}
S_\rhd(x_{B_1})\rhd x_{B_2}^{\rm op}\\
&\stackrel{\eqref{eq:btr}}{=}&(-1)^m\sum_{\pi:|\pi|=2,\,B_2\neq\emptyset}
x_{B_1}\btr x_{B_2}^{\rm op}\\
&\stackrel{\eqref{eq:btr_coalg}}{=}&(-1)^m\sum_\pi (x_{B_1}\btr x_{b_1})\cdots(x_{B_{|\pi|-1}}\btr x_{b_{|\pi|-1}}),
\end{eqnarray*}
where the sum in the last equality ranges over all partitions $\pi$ of $[m]$ under the stated condition. Note that each term $(x_{B_1}\btr x_{b_1})\cdots(x_{B_{|\pi|-1}}\btr x_{b_{|\pi|-1}})$ is a homogeneous polynomial in $\Ten(V)$ of degree $|\pi|-1$.
\end{proof}

In order to apply Eq.~\eqref{eq:antipode} for concrete computation, we still need an explicit combinatorial formula for the product $\btr$ on $\Ten(V)$. First we give the following key lemma.
\begin{lemma}\label{lem:btr_recursive}
For any letters $x_1,\dots,x_m\in V$ with $m\geq2$ and word $Y\in\Ten(V)$, the following recursive formula holds,
\begin{equation}\label{eq:btr_recursive}
x_1\cdots x_m \btr Y = x_2 \cdots x_m \btr (x_1 \rhd Y) + \sum_{i=2}^m x_2\cdots(x_1\rhd x_i)\cdots x_m \btr Y.
\end{equation}
\end{lemma}
\begin{proof}We obtain the stated recursive formula as follows,
\begin{eqnarray*}
x_1\cdots x_m \btr Y &\stackrel{\eqref{eq:btr}}{=}& (-1)^{m}S_\rhd(x_1\cdots x_m)\rhd Y\\
&\stackrel{\eqref{post-rbb-1}}{=}& (-1)^{m}S_\rhd(x_1\circ (x_2\cdots x_m)-x_1\rhd(x_2\cdots x_m))\rhd Y\\
&\stackrel{\eqref{Post-2}}{=}&
(-1)^{m}(S_\rhd(x_2\cdots x_m)\circ  S_\rhd(x_1))\rhd Y
+(-1)^{m+1}
S_\rhd\left(\sum_{i=2}^m x_2\cdots(x_1\rhd x_i)\cdots x_m\right)\rhd Y\\
&\stackrel{\eqref{Post-4},\,\eqref{post-rbb-1}}{=}&
(-1)^{m}S_\rhd(x_2\cdots x_m)
\rhd (S_\rhd(x_1)\rhd Y)
+(-1)^{m+1}\sum_{i=2}^m
S_\rhd(x_2\cdots(x_1\rhd x_i)\cdots x_m)\rhd Y\\
&\stackrel{\eqref{eq:antipo},\,\eqref{eq:btr}}{=}&
x_2\cdots x_m
\btr (x_1\rhd Y)
+\sum_{i=2}^m x_2\cdots(x_1\rhd x_i)\cdots x_m \btr Y.\hfill \qedhere
\end{eqnarray*}
\end{proof}

\begin{defn}\label{defn:nest}
For any letters $x_1,\dots,x_m\in V$, word $Y\in\Ten(V)$ and permutation $w\in S_m$ (with one-word notation), we define a specific element
$$[x_1,\dots,x_m;Y]_w$$
in $\Ten(V)$ by the following steps,
\begin{enumerate}[(i)]
\item\label{step1}
 Begin with the sequence $x_{w(1)}\ \cdots\ x_{w(m)}\ Y$.
 First use a pair of parentheses to enclose $x_1$ and its right adjacent letter $x_{w(w^{-1}(1)+1)}$ or right adjacent word $Y$, then use the product $\rhd$ to merge them.

\item\label{step2}
Next enclose the letter involving $x_2$ and its right adjacent letter, or right adjacent word involving $Y$ by a pair of parentheses, and then use the product $\rhd$ to merge them.

\item\label{step3}
Iteratively carry on such a procedure for letters involving $x_3,\dots, x_m$ successively, until there have been $m-1$ pairs of parentheses and $m$ symbols $\rhd$ inserted. The resulting element in $\Ten(V)$ is $[x_1,\dots,x_m;Y]_w$.
\end{enumerate}
\end{defn}

For example, we can obtain $[x_1,\dots,x_6;Y]_{256314}$ as follows,
\begin{eqnarray*}
x_2\ x_5\ x_6\ x_3\ x_1\ x_4\ Y &\stackrel{\rm\ref{step1}}{\rightsquigarrow}&
x_2\ x_5\ x_6\ x_3\ (x_1\rhd x_4)\ Y\\
&\stackrel{\rm\ref{step2}}{\rightsquigarrow}&
(x_2\rhd x_5)\ x_6\ x_3\ (x_1\rhd x_4)\ Y\\
&\stackrel{\rm\ref{step3}}{\rightsquigarrow}&
(x_2\rhd x_5)\ x_6\ (x_3\rhd (x_1\rhd x_4))\ Y\\
&\rightsquigarrow&
(x_2\rhd x_5)\ x_6\ ((x_3\rhd (x_1\rhd x_4))\rhd Y)\\
&\rightsquigarrow&
((x_2\rhd x_5)\rhd x_6)\ ((x_3\rhd (x_1\rhd x_4))\rhd Y)\\
&\rightsquigarrow&
((x_2\rhd x_5)\rhd x_6)\rhd((x_3\rhd (x_1\rhd x_4))\rhd Y)\\
&=& [x_1,\dots,x_6;Y]_{256314}.
\end{eqnarray*}

Now we are in the position to give the following combinatorial formula.
\begin{prop}\label{prop:btr}
For any letters $x_1,\dots,x_m\in V$ and word $Y\in \Ten(V)$,
\begin{equation}\label{eq:btr_comb}
x_1\cdots x_m \btr Y = \sum_{w\in S_m}[x_1,\dots,x_m;Y]_w.
\end{equation}
\end{prop}
\begin{proof}
Abbreviate $\sum_{w\in S_m}[x_1,\dots,x_m;Y]_w$ as $[x_1,\dots,x_m;Y]$.
When $m=1$, we clearly have
$$x_1 \btr Y = [x_1;Y]=x_1\rhd Y$$
by definition. For $m\geq2$,
\begin{eqnarray*}
[x_1,\dots,x_m;Y]
&=& \sum_{w\in S_m\atop w(m)=1}[x_1,\dots,x_m;Y]_w + \sum_{i=1}^{m-1}\sum_{w\in S_m\atop w(i)=1}[x_1,\dots,x_m;Y]_w
\\&=&
\sum_{\sigma\in S_{m-1}}[x_2,\dots, x_m;x_1\rhd Y]_\sigma\\
&+&
\sum_{j=2}^m \sum_{\sigma\in S_{m-1}}[x_2,\dots,x_{j-1},x_1\rhd x_j,x_{j+1},\dots, x_m;Y]_\sigma\\
&=&
[x_2,\dots,x_m; x_1\rhd Y]+\sum_{j=2}^m [x_2,\dots, x_{j-1},x_1\rhd x_j,x_{j+1},\dots,x_m;Y],
\end{eqnarray*}
where the second equality is due to the following fact. For any $i\in[m-1]$ and $w\in S_m$ with $w(i)=1$, there exists a unique $\sigma\in S_{m-1}$ such that
$$\sigma(k)=\begin{cases}
w(k)-1,&1\leq k<i,\\
w(k+1)-1,&i\leq k\leq m-1,
\end{cases}$$
and then $[x_1,\dots,x_m;Y]_w=[x_2,\dots,x_{j-1},x_1\rhd x_j,x_{j+1},\dots, x_m;Y]_\sigma$ for  $j=w(i+1)$ by Definition~\ref{defn:nest}. Hence, $[x_1,\dots,x_m;Y]$ satisfies the same recursive formula
as $x_1\cdots x_m\btr Y$ does in Eq.~\eqref{eq:btr_recursive}, and the stated equality~\eqref{eq:btr_comb} holds.
\end{proof}

\begin{remark}\label{rem:btr_env}
Given any post-Lie algebra $(\frakg,[\cdot,\cdot]_\frakg,\rhd)$, the product $\btr$ defined on $\Ten(\frakg)$ as in \eqref{eq:btr} can not be induced on $\U(\frakg)$ in general. Otherwise, we should have
$$J\btr \Ten(\frakg)\subseteq J,$$
where $J=(xy-yx-[x,y]_\frakg\,|\, x,y\in\frakg)$ is the Hopf ideal of $\Ten(\frakg)$ such that  $\U(\frakg)=\Ten(\frakg)/J$.

However, for any letters $x,y,z\in\frakg$, we check that
\begin{eqnarray*}
(xy-yx-[x,y]_\frakg)\btr z
&\stackrel{\eqref{eq:btr_recursive}}{=}&
y\rhd (x\rhd z) + (x\rhd y)\rhd z - x\rhd (y\rhd z)
- (y\rhd x)\rhd z -[x,y]_\frakg\rhd z \\
&=& y\rhd (x\rhd z) -x\rhd (y\rhd z) + (x\rhd y - y\rhd x -[x,y]_\frakg)\rhd z\\
&\stackrel{\eqref{Post-L-2}}{=}&   -2[x,y]_\frakg \rhd  z \in \frakg.
\end{eqnarray*}
By the Poincar\'e-Birkhoff-Witt theorem, we have $\frakg\cap J=\{0\}$, so $(xy-yx-[x,y]_\frakg)\btr z\in J$ if and only if $[x,y]_\frakg \rhd  z=0$, which does not necessarily hold.
\end{remark}

Next we generalize the antipode formula~\eqref{eq:antipode} to make it applicable for the universal enveloping algebra of an arbitrary post-Lie algebra, not just the free object over a magma algebra.
\begin{theorem}\label{thm:antipode_uea}
Given a post-Lie algebra $(\frakg,[\cdot,\cdot]_\frakg,\rhd)$, the antipode $S_\rhd$ of the sub-adjacent Hopf algebra $\U(\frakg)_\rhd$ is given by
\begin{equation}\label{eq:antipode_uea}
S_\rhd(x_1\cdots x_m)=(-1)^m\sum_\pi [x_{B_1}; x_{b_1}]\cdots[x_{B_{|\pi|-1}}; x_{b_{|\pi|-1}}],\quad \forall x_1,\dots, x_m\in \frakg,
\end{equation}
where the sum ranges over all partitions $\pi$ of $[m]$ into a tuple $(B_1,\dots,B_{|\pi|})$ of possibly empty subsets such that
$B_{|\pi|}=\{b_1>\cdots>b_{|\pi|-1}\}$. Also, we use the same abbreviation  $[x_I;y]\coloneqq\sum_{w\in S_r}[x_{i_1},\dots, x_{i_r};y]_w\in \frakg$ as in
the proof of Proposition~\ref{prop:btr} for any $I=\{i_1<\cdots< i_r\}\subseteq[m]$ and $y\in\frakg$. In particular, $[x_\emptyset;y]=y$.
\end{theorem}
\begin{proof}
First take $(\frakg,\rhd)$ as a magma algebra, and apply Proposition~\ref{prop:antipode} to the post-Hopf algebra $(\Ten(\frakg),\rhd)$, where the post-Hopf product of $\Ten(\frakg)$ is extended from the post-Lie product of $\frakg$.
Then we have the antipode formula~\eqref{eq:antipode} for its sub-adjacent Hopf algebra $\Ten(\frakg)_\rhd$.

Now let $J=(xy-yx-[x,y]_\frakg\,|\, x,y\in\frakg)$ be the Hopf ideal of $\Ten(\frakg)$ such that  $\U(\frakg)=\Ten(\frakg)/J$. Since $(\U(\frakg),\rhd)$ is a post-Hopf quotient algebra of $(\Ten(\frakg),\rhd)$, the antipode formula~\eqref{eq:antipode} can be induced on the sub-adjacent Hopf algebra $\U(\frakg)_\rhd$ as a quotient Hopf algebra of $\Ten(\frakg)_\rhd$. Though the product $\btr$ fails to be induced on $\U(\frakg)$ as pointed out in Remark~\ref{rem:btr_env}, we can interpret any factor $x_{B_i}\btr x_{b_i}$ in the RHS of~\eqref{eq:antipode} by~\eqref{eq:btr_comb} instead, and denote it as $[x_{B_i}; x_{b_i}]$. Then the desired formula is obtained.
\end{proof}

\begin{remark}
For any pre-Lie algebra $(\frakg,\rhd)$, $\U(\frakg)$ is the symmetric algebra $S(\frakg)$ over $\frakg$, and
the antipode formula \eqref{eq:antipode_uea} works for the sub-adjacent Hopf algebra $S(\frakg)_\rhd$. We take it as the graded-dual version of that for right-handed polynomial Hopf algebras given in \cite[Theorem~8]{MP}. However, it is unlikely to deduce any one of these two antipode formulas directly from the other.
\end{remark}

\section{Gavrilov's $K$-map and the Guin-Oudom isomorphism}\label{sec:OG}

In~\cite{Gav1} A. V. Gavrilov introduced the so-called $K$-map, which is crucial for building higher-order covariant derivatives on a smooth manifold endowed with an affine connection~(see also~\cite{Gav} and references therein).

From a post-Lie algebra point of view given in \cite{AEMM}, Gavrilov's $K$-map relates the two Hopf algebra structures on the universal enveloping algebra of the free post-Lie algebra over a magma algebra. We restate this result as follows.
\begin{prop}[{\cite[Theorem~3]{AEMM}}]
For the post-Hopf algebra $(\Ten(V),\cdot\,,\Delta_\shuffle,\rhd)$ given in Proposition~\ref{magma-to-post},
there exists a Hopf algebra isomorphism $K:\Ten(V)_\rhd\to \Ten(V)$ called {\bf the $K$-map} and recursively defined by
\begin{equation}\label{eq:K-map}
K(x)=x\quad \mbox{and}\quad K(x_1\cdots x_n)=x_1K(x_2\cdots x_n) - \sum_{i=2}^n K(x_2\cdots x_{i-1}(x_1\rhd x_i)x_{i+1}\cdots x_n)
\end{equation}
for any $x,x_1,\dots,x_n\in V$.
\end{prop}

The closed formula for the inverse $K^{-1}:\Ten(V)\to \Ten(V)_\rhd$ is given as follows (see~\cite[Prop.~6]{AEMM}).
\begin{equation}\label{eq:K-map-inv}
K^{-1}(x_1\cdots x_n)=x_1 \circ  \cdots \circ  x_n=\sum_\pi x^\rhd_{B_1}\cdots x^\rhd_{B_{|\pi|}},\quad\forall
x_1,\dots,x_n\in V,
\end{equation}
where the sum ranges over all set partitions  $\pi=(B_1,\dots,B_{|\pi|})$ of $[n]$ with blocks $B_1,\dots,B_{|\pi|}$ satisfying $\max B_1<\cdots<\max B_{|\pi|}$, and $x^\rhd_I\coloneqq x_{i_1}\rhd(x_{i_2}\rhd\cdots(x_{i_{r-1}}\rhd x_{i_r})\cdots)\in V$ for any $I=\{i_1<\cdots< i_r\}\subseteq[n]$.

However, as pointed out in \cite{AEMM}, the recursive
expression~\eqref{eq:K-map} does not deliver easily a closed formula for the $K$-map $K$ itself.
Here we utilize the twisted product $\btr$ in Eq.~\eqref{eq:btr} to present a closed formula for $K$ combinatorially.
\begin{prop}\label{prop:K-map'}
The $K$-map $K:\Ten(V)_\rhd\to \Ten(V)$ is given by
\begin{equation}\label{eq:K-map'}
K(x_1\cdots x_n)=\sum_\pi(-1)^{n-|\pi|} x^\btr_{B_1}\cdots x^\btr_{B_{|\pi|}},\quad
\forall x_1,\dots,x_n\in V,
\end{equation}
where the sum ranges over all set partitions  $\pi=(B_1,\dots,B_{|\pi|})$ of $[n]$ with blocks $B_1,\dots,B_{|\pi|}$ satisfying $\max B_1<\cdots<\max B_{|\pi|}$, and
$$x^\btr_I\coloneqq x_{i_1}\cdots x_{i_{r-1}}\btr x_{i_r}\stackrel{\eqref{eq:btr_comb}}{=}\sum_{w\in S_{r-1}}[x_{i_1},\dots, x_{i_{r-1}};x_{i_r}]_w\in V$$
for any $I=\{i_1<\cdots< i_r\}\subseteq[n]$, with the convention that $x^\btr_{\{i\}}=x_i$, $\forall i\in [n]$.
\end{prop}
\begin{proof}
Let $R(x_1\cdots x_n)$ be the RHS of Eq.~\eqref{eq:K-map'}, and we prove this equality by induction on $n$. When $n=1$ we clearly have $K(x_1)=R(x_1)=x_1$. For $n\geq2$, let $\pi=(B_1,\dots,B_{|\pi|})$ be any set partition of $[n]$ satisfying the stated condition. Then $B_1=\{1\}$, or otherwise $1\in B_k$ for some block $B_k$ such that $|B_k|\geq2$, and then Eq.~\eqref{eq:btr_recursive} implies that
$$ x^\btr_{B_k}=x_{i_1}\cdots x_{i_{r-1}} \btr x_{i_r} = x_{i_2} \cdots x_{i_{r-1}} \btr (x_1 \rhd x_{i_r}) + \sum_{j=2}^{r-1} x_{i_2}\cdots(x_1\rhd x_{i_j})\cdots x_{i_{r-1}} \btr x_{i_r},\eqno{(*)}$$
if $B_k=\{1=i_1<i_2<\cdots< i_r\}$. Hence, we have
\begin{eqnarray*}
R(x_1\cdots x_n) &=& \sum_\pi(-1)^{n-|\pi|} x^\btr_{B_1}\cdots x^\btr_{B_{|\pi|}}\\
&=& x_1\sum_{\pi\,:\,B_1=\{1\}}(-1)^{n-|\pi|} x^\btr_{B_2}\cdots x^\btr_{B_{|\pi|}}
- \sum_{\pi\,:\,B_1\neq\{1\}}(-1)^{n-1-|\pi|} x^\btr_{B_1}\cdots x^\btr_{B_{|\pi|}}\\
&\stackrel{(*)}{=}& x_1R(x_2\cdots x_n)  - \sum_{i=2}^n R(x_2\cdots x_{i-1}(x_1\rhd x_i)x_{i+1}\cdots x_n)\\
&=& x_1K(x_2\cdots x_n)- \sum_{i=2}^n K(x_2\cdots x_{i-1}(x_1\rhd x_i)x_{i+1}\cdots x_n)\\
&\stackrel{\eqref{eq:K-map}}{=}& K(x_1\cdots x_n),
\end{eqnarray*}
where the fourth equality is due to the induction hypothesis.
\end{proof}

Since the $K$-map $K$ is a Hopf algebra isomorphism from $\Ten(V)_\rhd$ to $\Ten(V)$, we can also compute the antipode $S_\rhd$ on $T(V)_\rhd$ via the equality
$$S_\rhd(x_1\cdots x_n)=K^{-1}(S(K(x_1\cdots x_n))),\quad
\forall x_1,\dots,x_n\in V.$$
Note that the calculation involves no cancellation when one applies Eqs.~\eqref{eq:K-map'}, \eqref{eq:antipo} and~\eqref{eq:K-map-inv} successively, and all monomials in $K^{-1}(S(K(x_1\cdots x_n)))$ have integer coefficients of the same sign $(-1)^n$.
However, such an approach is somewhat oblique.

\begin{exam}
We compute the case when $n=3$. There are 5 set partitions of $[3]$, namely $123$, $1|23$, $12|3$, $2|13$, $1|2|3$, and
\begin{eqnarray*}
&&S_\rhd(x_1x_2x_3)=K^{-1}(S(K(x_1x_2x_3)))\\
&\stackrel{\eqref{eq:K-map'}}{=}&
K^{-1}(S(x^\btr_1x^\btr_2x^\btr_3-x^\btr_1x^\btr_{23}
-x^\btr_{12}x^\btr_{3}- x^\btr_{2}x^\btr_{13}+ x^\btr_{123}))\\
&\stackrel{\eqref{eq:btr_comb}}{=}&
K^{-1}(S(x_1x_2x_3-x_1(x_2\rhd x_3)
-(x_1\rhd x_2)x_3- x_2(x_1\rhd x_3)+ (x_1\rhd x_2)\rhd x_3+ x_2\rhd(x_1\rhd x_3)))\\
&\stackrel{\eqref{eq:antipo}}{=}&
-K^{-1}(x_3x_2x_1+(x_2\rhd x_3)x_1
+x_3(x_1\rhd x_2)+ (x_1\rhd x_3)x_2+ (x_1\rhd x_2)\rhd x_3 + x_2\rhd(x_1\rhd x_3))\\
&\stackrel{\eqref{eq:K-map-inv}}{=}&
-x_3x_2x_1-(x_3\rhd x_2)x_1  -x_3(x_2\rhd x_1) -x_2(x_3\rhd x_1) -(x_2\rhd x_3)x_1
-x_3(x_1\rhd x_2) - (x_1\rhd x_3)x_2
\\
&&-x_3\rhd(x_2\rhd x_1)-(x_2\rhd x_3)\rhd x_1
-x_3\rhd (x_1\rhd x_2) - (x_1\rhd x_3)\rhd x_2 - (x_1\rhd x_2)\rhd x_3 - x_2\rhd(x_1\rhd x_3).
\end{eqnarray*}
So we get the same result as in Example~\ref{ex:anti}.
\end{exam}

\subsection{A closed inverse formula for the Guin-Oudom isomorphism}
Given a post-Lie algebra $(\frakg,[\cdot,\cdot]_\frakg,\rhd)$, the Guin-Oudom isomorphism $\phi:\U(\frakg_\rhd)\to\U(\frakg)_\rhd$ in Theorem~\ref{th:post-uea} is the unique extension of the identity map $\id:\frakg_\rhd\to\frakg$, and has exactly the same combinatorial expression as $K^{-1}:\Ten(\frakg)\to\Ten(\frakg)_\rhd$ in~\eqref{eq:K-map-inv} (see~\cite[Eq.~(18)]{EMM}). In fact, we have
\begin{prop}\label{prop:K-map-OG}
For a post-Lie algebra $(\frakg,[\cdot,\cdot]_\frakg,\rhd)$,
the inverse $K^{-1}:\Ten(\frakg)\to\Ten(\frakg)_\rhd$ induces the Guin-Oudom isomorphism $\phi:\U(\frakg_\rhd)\to \U(\frakg)_\rhd$,
\end{prop}
\begin{proof}
Let $J=(xy-yx-[x,y]_\frakg\,|\, x,y\in\frakg)$ and $J_\rhd=(xy-yx-[x,y]_{\frakg_\rhd}\,|\, x,y\in\frakg)$ be the Hopf ideals of $\Ten(\frakg)$ such that $\U(\frakg)=\Ten(\frakg)/J$ and $\U(\frakg_\rhd)=\Ten(\frakg)/J_\rhd$ respectively.

We only need to show that $K^{-1}(J_\rhd)\subseteq J$.  For any letters $x,y\in\frakg$ and words $X,Y\in \Ten(\frakg)$,
\begin{eqnarray*}
K^{-1}(xy-yx-[x,y]_{\frakg_\rhd})&\stackrel{\eqref{eq:K-map-inv}}{=}&
x\circ  y-y\circ  x-[x,y]_{\frakg_\rhd}\\
&\stackrel{\eqref{post-rbb-1},\,\eqref{eq:post-L-sub}}{=} & (xy + x\rhd y) - (yx+y\rhd x)-(x\rhd y - y\rhd x + [x,y]_\frakg) \\
&=& xy-yx-[x,y]_\frakg,
\end{eqnarray*}
and then
\begin{eqnarray*}
K^{-1}(X(xy-yx-[x,y]_{\frakg_\rhd})Y)
&=&K^{-1}(X)\circ  K^{-1}(xy-yx-[x,y]_{\frakg_\rhd})\circ  K^{-1}(Y)\\
&=&K^{-1}(X)\circ (xy-yx-[x,y]_\frakg)\circ  K^{-1}(Y)\in J,
\end{eqnarray*}
as we have known that $J$ is a Hopf ideal of $\Ten(V)_\rhd$.  Since $\phi$ is invertible, $K^{-1}(J_\rhd)= J$ indeed.
\end{proof}

Correspondingly, we obtain a closed inverse formula for the Guin-Oudom isomorphism $\phi$.
\begin{theorem}\label{thm:hopf-iso-inv}
The inverse $\phi^{-1}:\U(\frakg)_\rhd\to \U(\frakg_\rhd)$ of the Guin-Oudom isomorphism is given by
\begin{equation}\label{eq:hopf-iso-inv}
\phi^{-1}(x_1\cdots x_n)=\sum_\pi(-1)^{n-|\pi|} [x_{B_1}]\cdots [x_{B_{|\pi|}}],\quad
\forall x_1,\dots,x_n\in \frakg,
\end{equation}
where the sum ranges over all set partitions  $\pi=(B_1,\dots,B_{|\pi|})$ of $[n]$ with blocks $B_1,\dots,B_{|\pi|}$ satisfying $\max B_1<\cdots<\max B_{|\pi|}$, and
$[x_I]\coloneqq\sum_{w\in S_{r-1}}[x_{i_1},\dots, x_{i_{r-1}};x_{i_r}]_w\in \frakg$ for any $I=\{i_1<\cdots< i_r\}\subseteq[n]$, with the convention that $[x_{\{i\}}]=x_i$, $\forall i\in [n]$.
\end{theorem}
\begin{proof}
According to Proposition~\ref{prop:K-map-OG},
the closed formula~\eqref{eq:K-map'} for the $K$-map $K$ is also available for $\phi^{-1}$, but we  need to add the following comment on it.

In general, the product $\btr$ on $\Ten(\frakg)$ can not be induced on $\U(\frakg_\rhd)$ neither. Therefore, any element $x_{B_i}^\btr$ in~\eqref{eq:K-map'} should be interpreted by the combinatorial characterization~\eqref{eq:btr_comb} directly, and we rewrite it as $[x_{B_i}]$.
\end{proof}

\begin{remark}
Recall the Bell polynomial (see~\cite{ELM1} for details)
$$B_{n,k}(x_1,\dots,x_l)=\sum_{j_1+j_2+\cdots+j_l=k\atop
j_1+2j_2+\cdots+lj_l=n}\dfrac{n!}{j_1!\cdots j_l!}\left(\dfrac{x_1}{1!}\right)^{j_1}\cdots\left(\dfrac{x_l}{l!}\right)^{j_l}$$
with $l=n-k+1$, the coefficient
of $x_1^{j_1}\cdots x_l^{j_l}$
in which is equal to the number of set partitions of $[n]$ with $j_r$ blocks of size $r$ for any $r\in[l]$. In particular,
$$B_{n,k}(0!,1!,\dots,(l-1)!)=\sum_{j_1+j_2+\cdots+j_l=k\atop
j_1+2j_2+\cdots+lj_l=n}\dfrac{n!}{j_1!\cdots j_l!1^{j_1}\cdots l^{j_l}}$$
is exactly the number of permutations in $S_n$ with a factorization into $k$ disjoint cycles.

Hence, the number of terms of $\phi^{-1}(x_1\cdots x_n)$ when expanded as in Eq.~\eqref{eq:hopf-iso-inv} is
$$\sum_{k=1}^n B_{n,k}(0!,1!,\dots,(l-1)!)=n!,$$
Actually, it can also be seen directly by induction on $n$ from the recursive formula~\eqref{eq:K-map}.
\end{remark}

\begin{exam}
There are 15 set partitions of $[4]$ listed as follows,
\[\begin{array}{l}
1|2|3|4,\quad 12|3|4,\quad 2|13|4,\quad 2|3|14,\quad 1|23|4,\quad 1|3|24,\quad 1|2|34,\\[.5em]
12|34,\quad 13|24,\quad 23|14,\quad 1|234,\quad 2|134,\quad 3|124,\quad 123|4,\quad 1234.
\end{array}\]
Then Eq.~\eqref{eq:hopf-iso-inv} shows that the expansion of $\phi^{-1}(x_1x_2x_3x_4)$ has $4!=24$ terms as follows,
\begin{align*}
\phi&^{-1}(x_1x_2x_3x_4) =[x_1][x_2][x_3][x_4]\\
&\quad -[x_{12}][x_3][x_4]
-[x_2][x_{13}][x_4]-[x_2][x_3][x_{14}]-[x_1][x_{23}][x_4]-[x_1][x_3][x_{24}]-[x_1][x_2][x_{34}]\\
&\quad +[x_{12}][x_{34}]+[x_{13}][x_{24}]+[x_{23}][x_{14}]
+[x_1][x_{234}]+[x_2][x_{134}]+[x_3][x_{124}]+[x_{123}][x_4]-[x_{1234}]\\
&= x_1x_2x_3x_4
-(x_1\rhd x_2)x_3x_4 -x_2(x_1\rhd x_3)x_4 -x_2x_3(x_1\rhd x_4)\\
&\quad - x_1(x_2\rhd x_3)x_4 -x_1x_3(x_2\rhd x_4)-x_1x_2(x_3\rhd x_4) \\
&\quad+(x_1\rhd x_2)(x_3\rhd x_4)+(x_1\rhd x_3)(x_2\rhd x_4)+(x_2\rhd x_3)(x_1\rhd x_4)\\
&\quad+ x_1((x_2\rhd x_3)\rhd x_4)+ x_1(x_3\rhd (x_2\rhd x_4))+ x_2((x_1\rhd x_3)\rhd x_4)+ x_2(x_3\rhd (x_1\rhd x_4))\\
&\quad +x_3((x_1\rhd x_2)\rhd x_4)+ x_3(x_2\rhd (x_1\rhd x_4))
+((x_1\rhd x_2)\rhd x_3)x_4+ (x_2\rhd (x_1\rhd x_3))x_4\\
&\quad -((x_1\rhd x_2)\rhd x_3)\rhd x_4 - (x_1\rhd x_3)\rhd (x_2\rhd x_4) - (x_2\rhd (x_1\rhd x_3))\rhd x_4 \\
&\quad- (x_2\rhd x_3)\rhd (x_1\rhd x_4) - x_3\rhd ((x_1\rhd x_2)\rhd x_4) - x_3\rhd (x_2\rhd (x_1\rhd x_4)).
\end{align*}

\end{exam}

In the end of this section, we give a different perspective on the inverse of the Guin-Oudom isomorphism $\phi$.
According to Theorem~\ref{thm:subHopf}, $\U(\frakg)$ is a  $\U(\frakg)_\rhd$-module bialgebra via the action $\rhd$. Correspondingly, we can twist $\rhd$ by $\phi$ to define a $\U(\frakg_\rhd)$-module bialgebra action $\rightharpoonup$ on $\U(\frakg)$ by
$$X\rightharpoonup Y\coloneqq \phi(X)\rhd Y,\quad \forall X\in \U(\frakg_\rhd),\,Y\in \U(\frakg).$$
\begin{prop}\label{prop:rRB}
The inverse $\phi^{-1}:\U(\frakg)\to\U(\frakg_\rhd)$ is a relative Rota-Baxter operator with respect to $\rightharpoonup$ in the sense of \cite[Def.~3.1]{LST}. Namely, $\phi^{-1}$ is a coalgebra homomorphism satisfying
$$\phi^{-1}(X)\phi^{-1}(Y)=\phi^{-1}(X_1(\phi^{-1}(X_2)\rightharpoonup Y)),\quad \forall X,Y\in \U(\frakg).$$
\end{prop}
\begin{proof}
Applying $\phi$ to both sides of the stated relative Rota-Baxter identity, one can see that it is equivalent to Eq.~\eqref{post-rbb-1} as $\phi$ is an algebra isomorphism  from $\U(\frakg_\rhd)$ to $\U(\frakg)_\rhd$.
\end{proof}
\begin{remark}
In \cite[Remark 6.13]{EMQ}, the authors pointed out that the post-Lie Magnus expansion is a relative Rota-Baxter operator connecting two formal groups.
Proposition~\ref{prop:rRB} is a generalization of it in the context of Hopf algebras.
\end{remark}

\section{The Grossman-Larson Hopf algebra of ordered trees}\label{sec:GL}
In this section, we write down the cancellation-free antipode formula for the Grossman-Larson Hopf algebra of ordered trees via Eqs.~\eqref{eq:antipode} and \eqref{eq:btr_comb}.
First review the Grossman-Larson Hopf algebra of ordered trees defined in~\cite{GL}.
For the definition of ordered trees (also called planar rooted trees), see~\cite[page 573]{St}.

\subsection{Basic terminology}
Let $\huaO$ be the set of isomorphism classes of ordered trees, which is presented by
\[\huaO= \Big\{
		\scalebox{0.6}{\ab}, \scalebox{0.6}{\aabb},
		\scalebox{0.6}{\aababb}, \scalebox{0.6}{\aaabbb},
        \scalebox{0.6}{\aabababb},\scalebox{0.6}{\aaabbabb},
		\scalebox{0.6}{\aabaabbb}, \scalebox{0.6}{\aaababbb}, \scalebox{0.6}{\aaaabbbb},\scalebox{0.6}{\aababababb},
\scalebox{0.6}{\aaabbababb},\scalebox{0.6}{\aabaabbabb},
\scalebox{0.6}{\aababaabbb},\scalebox{0.6}{\aaababbabb},
\scalebox{0.6}{\aabaababbb},\scalebox{0.6}{\aaabbaabbb},
\scalebox{0.6}{\aaabababbb},\scalebox{0.6}{\aaaabbbabb},
\scalebox{0.6}{\aabaaabbbb},\scalebox{0.6}{\aaaababbbb},
\scalebox{0.6}{\aaaabbabbb},\scalebox{0.6}{\aaabaabbbb},\scalebox{0.6}{\aaaaabbbbb},\ldots
			\Big\}.
\]

Let $\bk\{\huaO\}$ be the free $\bk$-module generated by  $\huaO$. Denote $\Ten(\bk\{\huaO\})$ the tensor algebra over $\bk\{\huaO\}$. We identify $\Ten(\bk\{\huaO\})$ with the free $\bk$-algebra generated by $\huaO$, and interpret its basis of pure tensors of trees as the set of isomorphism classes of ordered forests denoted by $\huaF$. Namely, $\Ten(\bk\{\huaO\})=\bk\{\huaF\}$.
In particular, the unit $1$ is the empty forest $\emptyset$.

Define the {\bf left grafting operator} $\curvearrowright:\bk\{\huaO\}\otimes \bk\{\huaO\}\to \bk\{\huaO\}$ by
\begin{eqnarray}\label{eq:left-graft}
\tau\curvearrowright \omega=\sum_{s\in {\rm Nodes}(\omega)}\tau\curvearrowright_{s}\omega,\quad \forall \tau,\omega\in \huaO,
\end{eqnarray}
where $\tau\curvearrowright_{s}\omega$ is the ordered tree resulting from attaching the root of $\tau$ to  the node $s$ of the tree $\omega$ from the left. For example, we have
\begin{eqnarray*}
\scalebox{0.6}{\aabb}\curvearrowright \scalebox{0.6}{\aabababb}=\scalebox{0.6}{\aaabbabababb}+\scalebox{0.6}{\aaaabbbababb}+\scalebox{0.6}{\aabaaabbbabb}+\scalebox{0.6}{\aababaaabbbb}\,.
\end{eqnarray*}

Let $B^+:\bk\{\huaF\}\to \bk\{\huaO\}$ be the linear map producing an ordered tree $\tau$ from any ordered forest $\tau_1\cdots\tau_m$ by grafting the $m$ trees $\tau_1,\dots,\tau_m$ on a new root $\scalebox{0.6}{\ab}$ in order. In particular, set $B^+(1)=\scalebox{0.6}{\ab}$ by convention.
For example, we have
\begin{eqnarray*}
B^+(\scalebox{0.6}{\aabb\,\,\,\aababb})=\scalebox{0.6}{\aaabbaababbb}.
\end{eqnarray*}

Let $B^-:\bk\{\huaO\}\to \bk\{\huaF\}$  be the linear map
 producing an ordered forest from any ordered tree $\tau$ by removing its root.
For example, we have
\begin{eqnarray*}
B^-(\scalebox{0.6}{\aababaaabbbb})=\scalebox{0.6}{\ab\,\ab\,\,\aaabbb}.
\end{eqnarray*}

\begin{defn}[\cite{GL}]
Denote the Grossman-Larson Hopf algebra of ordered trees by $$\huaH_{GL}=(\bk\{\huaO\},\circ_{GL},\Delta_{GL},S_{GL}).$$
Therein, the Grossman-Larson product $\circ_{GL}$ on $\bk\{\huaO\}$ is defined by letting $\tau \circ_{GL} \omega$ be the partial sum of
$\tau_1\curvearrowright(\tau_2\curvearrowright\cdots(\tau_m\curvearrowright\omega)\cdots)$ containing those trees with all $\tau_1,\dots,\tau_m$ grafted {\it only} on $\omega$,
when $\tau,\,\omega\in \huaO$ such that $\tau=B^+(\tau_1\cdots\tau_m)$.

The coproduct $\Delta_{GL}$ on $\bk\{\huaO\}$ is given by
$$\Delta_{GL}(\tau)=\sum_{I\subseteq[m]}\tau^+_I\otimes\tau^+_{[m]\setminus I}$$
for any $\tau=B^+(\tau_1\cdots\tau_m)\in \huaO$, with the notation $\tau^+_I=B^+(\prod_{i\in I}\tau_i)$ for any ordered subset $I\subseteq[m]$.
\end{defn}

Recall Takeuchi's well-known recurrent antipode formula as follows.
\begin{theorem}[\cite{Ta}]\label{thm:ta}
Any graded, connected $\bk$-bialgebra $(B,m,u,\Delta,\vep)$ is a Hopf algebra with its antipode $S$ given by
\begin{equation}\label{eq:ta}
S =\sum_{i\ge 0} (-1)^i p^{\star i}=\sum_{i\ge 0} (-1)^i m^{(i-1)} p^{\otimes i} \Delta^{(i-1)},
\end{equation}
where $\star$ is the convolution product, $u$ is the unit map, $p\coloneqq\id-u\vep$, $m^{(-1)} \coloneqq u$, and $\Delta^{(-1)} \coloneqq \vep$.
\end{theorem}

Applying Takeuchi's antipode formula~\eqref{eq:ta} to the graded, connected bialgebra $(\bk\{\huaO\},\circ_{GL},\Delta_{GL})$, one obtains the following formula of $S_{GL}$ (see e.g.~\cite[Lemma~3.1]{Zh}),
$$
S_{GL}(\tau)=\sum_{\pi}(-1)^{|\pi|}\tau^+_{B_1}\circ_{GL}\cdots\circ_{GL}\tau^+_{B_{|\pi|}}
$$
for any $\tau=B^+(\tau_1\cdots\tau_m)\in \huaO$, where the sum ranges over all set partitions $\pi=(B_1,\dots,B_{|\pi|})$ of $[m]$.

\subsection{A cancellation-free antipode formula for the Grossman-Larson Hopf algebra}
In the rest of the paper, we aim to find a cancellation-free antipode formula for $\huaH_{GL}$.

The magma algebra $(\bk\{\huaO\},\curvearrowright)$ given by Eq.~\eqref{eq:left-graft} is the free magma algebra on one generator, namely a single root $\scalebox{0.6}{\ab}$. Foissy's construction tells us that the left grafting operator $\curvearrowright$ on $\bk\{\huaO\}$ extends to a product $\rhd$ on $\bk\{\huaF\}$, so that
$$(\bk\{\huaF\},\cdot\,,\Delta_\shuffle,S,\rhd)$$
is the universal enveloping algebra of the free post-Lie algebra on one generator, namely a post-Hopf algebra. See \cite{ELM,Foissy,ML} for more details about free post-Lie algebras and the universal enveloping algebras of them.

As indicated in \cite[Proposition~4.1]{OG} and \cite{MW}, we notice that
\begin{prop}\label{prop:GL_iso}
The sub-adjacent Hopf algebra $\bk\{\huaF\}_\rhd=(\bk\{\huaF\},{\circ },\Delta_\shuffle,S_\rhd)$ is isomorphic to the Grossman-Larson Hopf algebra $\huaH_{GL}$ of ordered trees via the root grafting operator $B^+$. Also, using the post-Hopf product $\rhd$, the multiplication $\circ $ can be given by
\begin{eqnarray*}
\huaX \circ  \huaY&=&B^-(\huaX\rhd B^+(\huaY)),\quad \forall \huaX,\huaY\in\huaF.
\end{eqnarray*}
Equivalently,
$$\tau \circ_{GL} \omega = B^-(\tau)\rhd \omega,\quad \forall \tau,\omega\in\huaO.$$
\end{prop}

By the definition~\eqref{eq:left-graft} of the left grafting operator on ordered trees, the combinatorial formula~\eqref{eq:btr_comb} for the product $\btr$ on $\bk\{\huaF\}$ can be expressed more concretely as follows.
\begin{prop}\label{prop:GL_graft}
Given any trees $\tau_1,\dots,\tau_m,\omega\in\huaO$, the product
$$\tau_1\cdots\tau_m\btr \omega=\sum_{w\in S_m}[\tau_1,\dots,\tau_m;\omega]_w$$
is equal to the sum of all left graftings with rootstock $\omega$ and then scions $\tau_1,\dots,\tau_m$ grafted in any order but under the constraint that
$\tau_i$ is left to $\tau_j$ for any $i>j$ whenever they are attached to the same node.
\end{prop}
\begin{proof}
We prove it by induction on the number $m$ of scions. It is clear when $m=1$, as $\tau_1\btr\omega=\tau_1\curvearrowright\omega$.
When $m>1$, we restate the recursive formula~\eqref{eq:btr_recursive} of the product $\btr$ on $\bk\{\huaF\}$ as
$$\tau_1\cdots \tau_m \btr \omega = \tau_2 \cdots \tau_m \btr (\tau_1 \curvearrowright \omega) + \sum_{i=2}^m \tau_2\cdots(\tau_1 \curvearrowright \tau_i)\cdots \tau_m \btr \omega.$$
By the definition \eqref{eq:left-graft} of $\curvearrowright$, it means that $\tau_1$ is attached to any node among $\tau_2,\dots,\tau_m,\omega$ first, and then if any one of $\tau_2,\dots,\tau_m$ is latter attached to the same node as $\tau_1$ does, it should be left to $\tau_1$. Moreover, each term in the above sum has only $m-1$ scions now, so the result follows by the induction hypothesis.
\end{proof}

\begin{exam}
\, $\scalebox{0.6}{\ab\,\aabb}\,\btr \scalebox{0.6}{\aabb} =
(\scalebox{0.6}{\ab}\curvearrowright \scalebox{0.6}{\aabb})\curvearrowright \scalebox{0.6}{\aabb}
+ \scalebox{0.6}{\aabb}\curvearrowright (\scalebox{0.6}{\ab}\curvearrowright \scalebox{0.6}{\aabb})
=\scalebox{0.6}{\aaabbababb}+\scalebox{0.6}{\aaaabbabbb}+2\,\scalebox{0.6}{\aaaabbbabb}
+2\,\scalebox{0.6}{\aaaaabbbbb}+\scalebox{0.6}{\aabaaabbbb}+\scalebox{0.6}{\aaabbaabbb}
+\scalebox{0.6}{\aaababbabb}+\scalebox{0.6}{\aaaababbbb}$.
\end{exam}

\begin{theorem}\label{thm:GL_antipode}
The antipode $S_{GL}$ of $\huaH_{GL}$ has the following cancellation-free formula,
\begin{equation}\label{eq:GL_antipode}
S_{GL}(\tau)=(-1)^m\sum_\pi B^+\big((\tau^-_{B_1}\btr \tau_{b_1})\cdots(\tau^-_{B_{|\pi|-1}}\btr \tau_{b_{|\pi|-1}})\big),\quad \forall \tau=B^+(\tau_1\cdots \tau_m)\in \huaO,
\end{equation}
where the sum ranges over all partitions $\pi$ of $[m]$ into a tuple $(B_1,\dots,B_{|\pi|})$ of possibly empty subsets such that
$B_{|\pi|}=\{b_1>\cdots>b_{|\pi|-1}\}$. Also, we use the notation $\tau^-_I=\tau_{i_1}\cdots \tau_{i_r}\in \huaF$ for any $I=\{i_1<\cdots< i_r\}\subseteq[m]$. In particular, $\tau^-_{\emptyset}=1$.
\end{theorem}
\begin{proof}
According to Proposition~\ref{prop:GL_iso}, $B^+$ is a Hopf algebra isomorphism from $\bk\{\huaF\}_\rhd$ to $\huaH_{GL}$, so $S_{GL}B^+=B^+S_\rhd$, namely $S_{GL}=B^+S_\rhd B^-$.
Hence, the desired formula of $S_{GL}$ is obtained as an equivalent form of Eq.~\eqref{eq:antipode} for $S_\rhd$.

Furthermore, the combinatorial formula~\eqref{eq:btr_comb} implies that each term $(\tau^-_{B_1}\btr \tau_{b_1})\cdots(\tau^-_{B_{|\pi|-1}}\btr \tau_{b_{|\pi|-1}})$ is homogeneous in $\bk\{\huaF\}$, of degree $|\pi|-1$ and with positive integer coefficients. So Eq.~\eqref{eq:GL_antipode} is a cancellation-free formula for $S_{GL}$.
\end{proof}

\begin{remark}
The graded Hopf dual of  $\bk\{\huaF\}_\rhd$ is the Munthe-Kaas-Wright Hopf algebra $$\huaH_N=(\bk\{\huaF\},m_\shuffle,\Delta_N,S_N)$$
of ordered forests, as the natural setting of the Lie-Butcher series. In \cite[Prop.~3]{MW}, the authors gave a nonrecursive formula for the antipode $S_N$ of $\huaH_N$. Our closed formula~\eqref{eq:GL_antipode} of $S_{GL}$ serves as its graded-dual version.
\end{remark}

At last we give a simple illustration how to compute the antipode $S_{GL}$ via Eq.~\eqref{eq:GL_antipode}.
\begin{exam}
Let $\tau=\scalebox{0.6}{\aababaabbb}$. That is, $\tau_1=\tau_2=\scalebox{0.6}{\ab}$, $\tau_3=\scalebox{0.6}{\aabb}\,$ such that $\tau=B^+(\tau_1\tau_2\tau_3)$.

Totally $10$ appropriate partitions of $[3]$ with the corresponding sum of trees in the RHS of~\eqref{eq:GL_antipode} are listed as follows.
\[\begin{array}{llll}
(\emptyset,\emptyset,\emptyset,\{3,2,1\}) &  (1\btr\tau_3)(1\btr\tau_2)(1\btr\tau_1)=\tau_3\tau_2\tau_1
&\rightsquigarrow& \scalebox{0.6}{\aaabbababb}\\
(\{1\},\emptyset,\{3,2\}) &
(\tau_1\btr\tau_3)(1\btr\tau_2)=(\tau_1\curvearrowright\tau_3)\tau_2
&\rightsquigarrow&  \scalebox{0.6}{\aaaabbbabb}+\scalebox{0.6}{\aaababbabb} \\[.3em]
(\emptyset,\{1\},\{3,2\}) &
(1\btr\tau_3)(\tau_1\btr\tau_2)=\tau_3(\tau_1\curvearrowright\tau_2)  &\rightsquigarrow &\scalebox{0.6}{\aaabbaabbb}\\
(\{2\},\emptyset,\{3,1\}) &
(\tau_2\btr\tau_3)(1\btr\tau_1)=(\tau_2\curvearrowright\tau_3)\tau_1
&\rightsquigarrow  &\scalebox{0.6}{\aaaabbbabb}+\scalebox{0.6}{\aaababbabb}\\[.3em]
(\emptyset,\{2\},\{3,1\}) &
(1\btr\tau_3)(\tau_2\btr\tau_1)=\tau_3(\tau_2\curvearrowright\tau_1) & \rightsquigarrow &\scalebox{0.6}{\aaabbaabbb}\\
(\{3\},\emptyset,\{2,1\}) &
(\tau_3\btr\tau_2)(1\btr\tau_1)=(\tau_3\curvearrowright\tau_2)\tau_1
&\rightsquigarrow & \scalebox{0.6}{\aaaabbbabb}\\
(\emptyset,\{3\},\{2,1\}) &
(1\btr\tau_2)(\tau_3\btr\tau_1)=\tau_2(\tau_3\curvearrowright\tau_1)
&\rightsquigarrow & \scalebox{0.6}{\aabaaabbbb}\\
(\{1,2\},\{3\}) &
\tau_1\tau_2\btr\tau_3=(\tau_1\curvearrowright\tau_2)\curvearrowright\tau_3
+\tau_2\curvearrowright(\tau_1\curvearrowright\tau_3)
&\rightsquigarrow & 2\,\scalebox{0.6}{\aaaabbabbb}+2\,\scalebox{0.6}{\aaabaabbbb}+2\,\scalebox{0.6}{\aaaaabbbbb}
+\scalebox{0.6}{\aaabababbb}+\scalebox{0.6}{\aaaababbbb}\\
(\{1,3\},\{2\}) &
\tau_1\tau_3\btr\tau_2=(\tau_1\curvearrowright\tau_3)\curvearrowright\tau_2
+\tau_3\curvearrowright(\tau_1\curvearrowright\tau_2)
&\rightsquigarrow  & \scalebox{0.6}{\aaaabbabbb}+2\,\scalebox{0.6}{\aaaaabbbbb}+\scalebox{0.6}{\aaaababbbb}\\
(\{2,3\},\{1\}) &
\tau_2\tau_3\btr\tau_1=(\tau_2\curvearrowright\tau_3)\curvearrowright\tau_1
+\tau_3\curvearrowright(\tau_2\curvearrowright\tau_1)
&\rightsquigarrow  & \scalebox{0.6}{\aaaabbabbb}+2\,\scalebox{0.6}{\aaaaabbbbb}+\scalebox{0.6}{\aaaababbbb}.
\end{array}\]
Now adding up all of them, we have
$$S_{GL}(\scalebox{0.6}{\aababaabbb})= -\,\scalebox{0.6}{\aaabbababb}
-3\,\scalebox{0.6}{\aaaabbbabb}-\scalebox{0.6}{\aabaaabbbb}
-2\scalebox{0.6}{\aaababbabb}-2\,\scalebox{0.6}{\aaabbaabbb}
-4\,\scalebox{0.6}{\aaaabbabbb}-2\,\scalebox{0.6}{\aaabaabbbb}-6\,\scalebox{0.6}{\aaaaabbbbb}
-\scalebox{0.6}{\aaabababbb}-3\scalebox{0.6}{\aaaababbbb}.$$
\end{exam}

\bigskip
In the end, we give a conclusion. The main contributions of this paper include:
\begin{enumerate}[1.]
\item
the derivation of an explicit combinatorial formula for the antipode of the sub-adjacent Hopf algebra $\U(\frakg)_\rhd$ associated to a post-Lie algebra $(\frakg,[\cdot,\cdot]_\frakg,\rhd)$;
\item
giving a closed inverse formula for the Guin-Oudom isomorphism, or equivalently expressing monomials in $\U(\frakg)$ in terms of the Grossman-Larson product $\circ $ conversely;
\item
applying these constructions to obtain a cancellation-free antipode formula for the Grossman-Larson Hopf algebra of ordered trees with a tree-grafting interpretation.
\end{enumerate}
In subsequent researches, we will explore potential applications of our result to the study of formal post-Lie groups~\cite{BGST} and also the post-Lie Magnus expansion~\cite{EMQ}.

\bigskip
 \noindent
{\bf Acknowledgements.} We would like to thank the referee for helpful comments. This work is supported by National Natural Science Foundation of China (12071094, 12171155), and Basic and Applied Basic Research Foundation of Guangdong Province (2026A1515012750).

\bibliographystyle{amsplain}

\end{document}